\documentclass[12pt]{article}
\usepackage{amsfonts}

\usepackage[all]{xy}

\def\squarebox#1{\hbox to #1{\hfill\vbox to #1{\vfill}}}
\newcommand{\qed}{\hspace*{\fill}
\vbox{\hrule\hbox{\vrule\squarebox{.667em}\vrule}\hrule}\smallskip}
\newtheorem{teorema}{Theorem}[section]
\newtheorem{lema}[teorema]{Lemma}

\newtheorem{proposicao}[teorema]{Proposition}

\newenvironment{proof}{\noindent {\bf Proof:}}{\hfill $\qed $ \newline}

\begin{document}

\title{Entropy and its variational principle for noncompact metric spaces}
\author{Mauro Patr\~{a}o}
\maketitle

\begin{abstract}
In the present paper, we introduce a natural extension of
AKM-topological entropy for noncompact spaces and prove a
variational principle which states that the topological entropy,
the supremum of the measure theoretical entropies and the minimum
of the metric theoretical entropies always coincide. We apply the
variational principle to show that the topological entropy of
automorphisms of simply connected nilpotent Lie groups always
vanishes. This shows that the classical formula for the entropy of
an automorphism of a noncompact Lie group is just an upper bound
for its topological entropy.
\end{abstract}



\section{Introduction}

Topological entropy was introduced in \cite{akm} by Adler, Konheim
and MacAndrew (AKM) to study the dynamic of a continuous map $T :
X \to X$ defined on a compact space $X$. They have also
conjectured the well known variational principle which states that
\[
h(T) = \sup_{\mu} h_{\mu}(T),
\]
where $h(T)$ is the topological entropy of $T$, $h_{\mu}(T)$ is
the $\mu$-entropy of $T$ and the supremum is taken over all
$T$-invariant probabilities $\mu$ on $X$.

R. Bowen has extended in \cite{bowen} the concept of topological
entropy for noncompact metric spaces. This approach uses the
concept of generator sets, which are defined by means of a given
metric. For compact spaces this definition of entropy is, in fact,
independent of the metric and, indeed, coincides with the
AKM-topological entropy.

However, for noncompact spaces Bowen's entropy depends on the
metric. In Lemma 1.5 of \cite{handel}, it was proved the following
variational principle for locally compact spaces
\[
\sup_{\mu} h_{\mu}(T) = \inf_{d} h_{d}(T),
\]
where $h_d(T)$ is the $d$-entropy of $T$ and the infimum is taken
over all metric $d$ on $X$.

In the present paper, we improve this result showing that the
above infimum is always attained at $d$, whenever $d$ is a metric
satisfying some special conditions. Theses metrics are called
admissible metrics and always exist when $X$ is a locally compact
separable space. Moreover the supremum and the minimum coincide
with a natural extension of AKM-topological entropy, introduced in
the second section. In the third section, this variational
principle is proven.

In the last section, we apply the variational principle to
determine the topological entropy of automorphisms of some Lie
groups. First we consider linear isomorphisms of a finite
dimensional vector space and characterize their recurrent sets in
terms of their multiplicative Jordan decompositions. We show that
the topological entropy of these linear isomorphisms always
vanishes. This shows that the classical formula for its
$d$-entropy, where $d$ is the euclidian metric, is just an upper
bound for its null topological entropy. Remember that this
classical formula is given by
\[
h_d(T) = \sum_{|\lambda| > 1} \log |\lambda|,
\]
where $\lambda$ runs through the eigenvalues of the linear
isomorphism $T$ (cf. Theorem 8.14 in \cite{walters}). Thus it
might be an interesting problem to calculate the topological
entropy of a given automorphism $\phi$ of a general noncompact Lie
group $G$, since the classical formula for its $d$-entropy, where
$d$ is some invariant metric, is as well just an upper bound for
its topological entropy. We conclude this paper providing an
answer for this problem when $G$ is a simply connected nilpotent
Lie group
. We show that the topological entropy of its
automorphisms also vanishes.

\section{Topological entropy and admissible metrics}\label{sectopent}

We start this section presenting an extension of the
AKM-topological entropy introduced in \cite{akm}. Let $X$ be a
topological space and $T : X \to X$ be a proper map, i.e., $T$ is
a continuous map such that the pre-image by $T$ of any compact set
is compact. An \emph{admissible} covering of $X$ is an open and
finite covering $\alpha$ of $X$ such that, for each $A \in
\alpha$, the closure or the complement of $A$ is compact. Given an
admissible covering $\alpha$ of $X$, for every $n \in \mathbb{N}$,
we have that the set given by
\[
\alpha^n = \{A_0 \cap T^{-1}(A_1) \cap \ldots \cap T^{-n}(A_n) :
A_i \in \alpha\}
\]
is also an admissible covering of $X$, since $T$ is a proper map.
Given an admissible covering $\alpha$ of $X$, we denote by
$N(\alpha^n)$ the smallest cardinality of all sub-coverings of
$\alpha^n$. Exactly as in the compact case, it can be shown that
the sequence $\log N(\alpha^n)$ is subadditive, which implies the
existence of the following limit
\[
h(T, \alpha) = \lim_{n \to \infty} \frac{1}{n} \log N(\alpha^n).
\]
The \emph{topological entropy of the map $T$} is thus defined as
\[
h(T) = \sup_{\alpha} h(T, \alpha),
\]
where the supremum is taken over all admissible coverings $\alpha$
of $X$. We note that, when $X$ is already compact, the above
definition coincides with the AKM definition, since thus every
continuous map is proper and every open and finite covering of $X$
is admissible.

The following result generalizes, for non-compact spaces, the
relation of the topological entropies of two semi-conjugated maps.

\begin{proposicao}\label{propsemiconjugacao}
Let $T : X \to X$ and $S : Y \to Y$ be two proper maps, where $X$
and $Y$ are topological spaces. If $f : Y \to X$ is a proper
surjective map such that $f \circ S = T \circ f$, then we have
that $h(T) \leq h(S)$.
\end{proposicao}
\begin{proof}
Let $\alpha$ be an admissible covering of $X$.  Since $f$ is a
proper map, it follows that
\[
f^{-1}(\alpha) = \{f^{-1}(A) : A \in \alpha\}
\]
is an admissible covering of $Y$. We claim that if $\beta$ is a
subset of $\alpha^n$, then $f^{-1}(\beta)$ is a subset of
$f^{-1}(\alpha)^n$. In fact, $B \in \beta$ if and only if
\[
B = A_0 \cap T^{-1}(A_1) \cap \ldots \cap T^{-n}(A_n)
\]
where $A_i \in \alpha$, for each $i \in \{0, \ldots, n\}$. Thus we
have that
\begin{eqnarray}
f^{-1}(B) & = & f^{-1}(A_0) \cap f^{-1}(T^{-1}(A_1)) \cap \ldots \cap f^{-1}(T^{-n}(A_n)) \nonumber \\
& = & f^{-1}(A_0) \cap S^{-1}(f^{-1}(A_1)) \cap \ldots \cap
S^{-n}(f^{-1}(A_n)), \nonumber
\end{eqnarray}
where we used that $f^{-1} \circ T^{-i} = S^{-i} \circ T^{-1}$,
since $f \circ S = T \circ f$. Thus it follows that $f^{-1}(\beta)
\subset f^{-1}(\alpha)^n$. Reciprocally, we claim that if $\gamma$
is a subset of $f^{-1}(\alpha)^n$, then $\gamma = f^{-1}(\beta)$,
where $\beta$ is some subset of $\alpha^n$. Proceeding
analogously, if $C \in \gamma$, then
\begin{eqnarray}
C & = & f^{-1}(A_0) \cap S^{-1}(f^{-1}(A_1)) \cap \ldots \cap S^{-n}(f^{-1}(A_n)) \\
& = & f^{-1}(A_0 \cap T^{-1}(A_1) \cap \ldots \cap T^{-n}(A_n)),
\nonumber
\end{eqnarray}
where $A_i \in \alpha$, for each $i \in \{0, \ldots, n\}$. Thus it
follows that $C = f^{-1}(B)$, for some $B \in \alpha^n$, which
implies that $\gamma = f^{-1}(\beta)$, where $\beta$ is some
subset of $\alpha^n$.

If $\beta$ is a sub-covering of $\alpha^n$, then $f^{-1}(\beta)$
is a sub-covering of $f^{-1}(\alpha)^n$. Reciprocally, if $\gamma$
is a sub-covering of $f^{-1}(\alpha)^n$, then $\gamma =
f^{-1}(\beta)$, where $\beta$ is some sub-covering of $\alpha^n$.
In fact, we have already known that $\gamma = f^{-1}(\beta)$, for
some subset $\beta$ of $\alpha^n$. We have to show just that
$\beta$ is a covering of $X$. But this is immediate, since $\gamma
= f^{-1}(\beta)$ is a covering of $Y$, $f$ is surjective and $B =
f(f^{-1}(B))$.

Hence we have that $N(\alpha^n) = N(f^{-1}(\alpha)^n)$ and taking
logarithms, dividing by $n$ and taking limits, it follows that
\[
h(T, \alpha) \leq h(S, f^{-1}(\alpha)) \leq h(S).
\]
Since $\alpha$ is an arbitrary admissible covering of $X$, we get
that $h(T) \leq h(S)$.
\end{proof}

Now we introduce the concept of entropy associated to some given
metric in two different ways. First we remember the classical
definition, introduced in \cite{bowen}. Given a metric space $(X,
d)$ and a continuous map $T : X \to X$, we consider the metric
given by
\[
d_n(x, y) = \max \{ d(T^i(x), T^i(y)) : 0 \leq i \leq n \},
\]
which is equivalent to the original metric $d$. Given a subset $Y
\subset X$, a subset $G$ is an $(n, \varepsilon)$-generator of $Y$
if and only if, for every $y \in Y$, there exists $x \in G$ such
that $d_n(x, y) < \varepsilon$. In other words, the collection of
all $\varepsilon$-balls of $d_n$ centered at points of $G$ is in
fact a covering of $Y$. We denote by $G_n(\varepsilon, Y)$ the
smallest cardinality of all $(n, \varepsilon)$-generators of $Y$
and define
\[
g(\varepsilon, Y) = \lim_{n \to \infty} \sup \frac{1}{n} \log
G_n(\varepsilon, Y).
\]
It can be proved that $g(\varepsilon, Y)$ is monotone with respect
to $\varepsilon$, so we can define
\[
h_d(T, Y) = \lim_{\varepsilon \downarrow 0} g(\varepsilon, Y).
\]
The \emph{$d$-entropy of the map $T$} is thus defined as
\[
h_d(T) = \sup_{K} h_d(T, K),
\]
where the supremum is taken over all compact subsets $K$ of $X$.

The definition of $d$-entropy which we present in this paper is
given by
\[
h^d(T) = \sup_{Y} h_d(T, Y),
\]
where now the supremum is taken over all subsets $Y$ of $X$
instead of just the compact ones. Note that here the metric $d$
lies at the superscript. Since we have that $h_d(T, Y)$ is
monotone with respect to $Y$, it follows that
\[
h^d(T) = h_d(T, X).
\]
We immediately notice that always $h_d(T) \leq h^d(T)$ and, when
$X$ is itself compact, we get the equality. Denoting
$G_n(\varepsilon, X)$ and $g(\varepsilon, X)$, respectively, by
$G_n(\varepsilon)$ and $g(\varepsilon)$, we get that
\[
h^d(T) = \lim_{\varepsilon \downarrow 0} g(\varepsilon)
\]
and that
\[
g(\varepsilon) = \lim_{n \to \infty} \sup \frac{1}{n} \log
G_n(\varepsilon).
\]

Our first result is a generalization of the well know result which
states that, if $(X, d)$ is compact, then the topological and the
metric entropies coincide. In the non-compact case, this remains
true but we need to ask for some regularity in the metric. Let
$(X, d)$ be a metric space. The metric $d$ is \emph{admissible} if
the following conditions are verified:
\begin{itemize}
\item[(1)] If $\alpha_{\delta} = \{B(x_1, \delta), \ldots, B(x_k,
\delta)\}$ is a covering of $X$, for every $\delta \in (a, b)$,
where $0 < a < b$, then there exists $\delta_{\varepsilon} \in (a,
b)$ such that $\alpha_{\delta_{\varepsilon}}$ is admissible.

\item[(2)] Every admissible covering of $X$ has a Lebesgue number.
\end{itemize}
We observe that, if $(X, d)$ is compact, then $d$ is automatically
admissible.

\begin{proposicao}\label{propequiv}
Let $T : X \to X$ be a proper map, where $(X, d)$ is a metric
space. If $d$ is an admissible metric, then $h(T) = h^d(T)$.
\end{proposicao}
\begin{proof}
First we claim that, for any admissible covering $\alpha$ of $X$,
it follows that $g(|\alpha|) \leq h(T, \alpha)$, where $|\alpha|$
is the maximum of the diameters of $A \in \alpha$. In fact, the
elements of $\alpha^n$ are given by $A_0 \cap T^{-1}(A_1) \cap
\ldots \cap T^{-n}(A_n)$, where $A_i \in \alpha$. Let $\beta$ be a
sub-covering of $\alpha^n$. For each $B \in \beta$, take an $x \in
B$ and consider $G$ the set of all such points. We claim that $G$
is an $(n, |\alpha|)$-generator set of $X$. In fact, let $y \in X$
and take some $A_0 \cap T^{-1}(A_1) \cap \ldots \cap T^{-n}(A_n)
\in \beta$ containing $y$. Taking $x \in G$ such that $x \in A_0
\cap T^{-1}(A_1) \cap \ldots \cap T^{-n}(A_n)$, we have that
$d(T^i(x), T^i(y)) < |\alpha|$, for every $i \in \{0, 1, \ldots,
n\}$, since $T^i(x), T^i(y) \in A_i$. Hence, for each sub-covering
$\beta$ of $\alpha^n$, there exists an $(n, |\alpha|)$-generator
set of $X$ such that $G_n(|\alpha|) \leq \# G \leq \# \beta$. Thus
we get that $G_n(|\alpha|) \leq N(\alpha^{n})$. Taking logarithms,
dividing by $n$ and taking limits, it follows as claimed that
$g(|\alpha|) \leq h(T, \alpha)$.

Now we claim that, for all $\varepsilon > 0$, there exists
$\delta_{\varepsilon} \in (\varepsilon/2, \varepsilon)$ such that
$h(T, \alpha_{\delta_{\varepsilon}}) \leq g(\varepsilon/2)$, where
$\alpha_{\delta_{\varepsilon}}$ is an admissible covering of balls
with radius equals to $\delta_{\varepsilon}$. In fact, if $G =
\{x_1, \ldots, x_k\}$ is an $(n, \varepsilon/2)$-generator of $X$,
for every $\delta \in (\varepsilon/2, \varepsilon)$, we have that
\[
\beta_{\delta} = \{B(x_i, \delta) \cap \ldots \cap
T^{-n}(B(T^n(x_i), \delta)) : x_i \in G \}
\]
is a covering of $X$. To see this, given $x \in X$, there exists
$x_i \in G$ such that $d_n(x, x_i) < \varepsilon/2 < \delta$ and
thus we have that
\[
x \in B(x_i, \delta) \cap \ldots \cap T^{-n}(B(T^n(x_i), \delta)).
\]
This implies that
\[
\alpha_{\delta} = \{B(T^l(x_i), \delta) : x_i \in G, 0 \leq l \leq
n\}
\]
is a finite covering of $X$, for every $\delta \in (\varepsilon/2,
\varepsilon)$. Since $d$ is an admissible metric, there exists
$\delta_{\varepsilon} \in (\varepsilon/2, \varepsilon)$ such that
$\alpha_{\delta_{\varepsilon}}$ is an admissible covering of balls
with radius equals to $\delta_{\varepsilon}$. Furthermore we have
that $\beta_{\delta_{\varepsilon}}$ is a sub-covering of
$\alpha_{\delta_{\varepsilon}}^n$. Hence, for each $(n,
\varepsilon/2)$-generator $G$ of $X$, there exist
$\delta_{\varepsilon} \in (\varepsilon/2, \varepsilon)$ and a
sub-covering $\beta_{\delta_{\varepsilon}}$ of
$\alpha_{\delta_{\varepsilon}}^n$ such that
$N(\alpha_{\delta_{\varepsilon}}) \leq \#
\beta_{\delta_{\varepsilon}} \leq \# G$. Thus it follows that
$N(\alpha_{\delta_{\varepsilon}}) \leq G_n(\varepsilon/2)$. Taking
logarithms, dividing by $n$ and taking limits, we get as claimed
that there exists $\delta_{\varepsilon} \in (\varepsilon/2,
\varepsilon)$ such that $h(T, \alpha_{\delta_{\varepsilon}}) \leq
g(\varepsilon/2)$. Since $|\alpha_{\delta_{\varepsilon}}| \leq
2\delta_{\varepsilon} \leq 2\varepsilon$, we have that
\[
g(2\varepsilon) \leq g(|\alpha_{\delta_{\varepsilon}}|) \leq h(T,
\alpha_{\delta_{\varepsilon}}) \leq g(\varepsilon/2).
\]
Taking limits with $\varepsilon \downarrow 0$, it follows that
\[
h^d(T) = \lim_{\varepsilon \downarrow 0} h(T,
\alpha_{\delta_{\varepsilon}}) = \sup_{\varepsilon > 0} h(T,
\alpha_{\delta_{\varepsilon}}).
\]

In order to complete the proof, it remains to show that the above
supremum is equal to $h(T)$. For any admissible covering $\alpha$
of $X$, take $\varepsilon$ a Lebesgue number of this covering. We
claim that $N(\alpha^n) \leq N(\alpha_{\delta_{\varepsilon}}^n)$,
where $\alpha_{\delta_{\varepsilon}}$ is given above. In fact,
since every element of $\alpha_{\delta_{\varepsilon}}^n$ if given
by $B_0 \cap \ldots \cap T^{-n}(B_n)$, where $\{B_0, \ldots,
B_n\}$ are balls of radius $\delta_{\varepsilon} < \varepsilon$,
there exist $\{A_0, \ldots, A_n\} \subset \alpha$ such that $B_i
\subset A_i$, for all $i \in \{0, \ldots, n\}$. Thus we have that
\[
B_0 \cap \ldots \cap T^{-n}(B_n) \subset A_0 \cap \ldots \cap
T^{-n}(A_n),
\]
showing that, for each sub-covering $\beta$ of
$\alpha_{\delta_{\varepsilon}}^n$, there exists a sub-covering
$\gamma$ of $\alpha^n$ such that $N(\alpha^n) \leq \# \gamma \leq
\# \beta$. Hence we get as claimed that $N(\alpha^n) \leq
N(\alpha_{\delta_{\varepsilon}}^n)$. Taking logarithms, dividing
by $n$ and taking limits, it follows that
\[
h(T, \alpha) \leq h(T, \alpha_{\delta_{\varepsilon}}) \leq
\sup_{\varepsilon > 0} h(T, \alpha_{\delta_{\varepsilon}}),
\]
which shows that
\[
h(T) = \sup_{\varepsilon > 0} h(T, \alpha_{\delta_{\varepsilon}}),
\]
completing the proof.
\end{proof}

From now one we assume that $X$ is a locally compact space. Thus
we have associated to $X$ its one point compactification, which we
will denote by $\widetilde{X}$. We have that $\widetilde{X}$ is
defined as the disjoint union of $X$ with $\{\infty\}$, where
$\infty$ is some point not in $X$ called the \emph{point at the
infinity}. The topology in $\widetilde{X}$ consists by the former
open sets in $X$ and by the sets $A \cup \{\infty\}$, where the
the complement of $A$ in $X$ is compact. If $T : X \to X$ is a
proper map, defining $\widetilde{T} : \widetilde{X} \to
\widetilde{X}$ by
\[
\widetilde{T}(\widetilde{x}) = \left\{
\begin{array}{ll}
T(\widetilde{x}), &  x \neq \infty \\
\infty, & x = \infty
\end{array}
\right.,
\]
we have that $\widetilde{T}$ is also proper map, called the
\emph{extension of $T$ to $\widetilde{X}$}. To see this, we only
need to verify that $\widetilde{T}$ is continuous at $\infty$. If
$A \cup \{\infty\}$ is a neighborhood of $\infty$, then the
complement of $A$ is compact and we have that
\[
\widetilde{T}^{-1}(A \cup \{\infty\}) = T^{-1}(A) \cup \{\infty\}
\]
is also a neighborhood of $\infty$, since $T$ is proper and thus
the complement of $T^{-1}(A)$ is also compact.

The following result shows that the restriction to $X$ of any
metric on $\widetilde{X}$ is always admissible and that their
respective entropies coincide.

\begin{proposicao}\label{propadmissivel}
Let $T : X \to X$ be a proper map, where $X$ is a locally compact
separable space. Let $d$ be the metric given by the restriction to
$X$ of some metric $\widetilde{d}$ on $\widetilde{X}$, the one
point compactification of $X$. Then it follows that $d$ is an
admissible metric and that
\[
h^d(T) = h^{\widetilde{d}}\left(\widetilde{T}\right),
\]
where $\widetilde{T}$ is the extension of $T$ to $\widetilde{X}$.
In particular, we get that $h(T) = h\left(\widetilde{T}\right)$.
\end{proposicao}
\begin{proof}
First we show that $G_n(\varepsilon)$ is finite for every
$\varepsilon > 0$. Let $\widetilde{G} = \{\widetilde{x}_1, \ldots,
\widetilde{x}_k\} \subset \widetilde{X}$ be an $(n,
\varepsilon/2)$-generator set of $\widetilde{X}$. By the density
of $X$ in $\widetilde{X}$, it follows that there exist $\{x_1,
\ldots, x_k\} \subset X$, such that $\widetilde{d}_n(x_i,
\widetilde{x}_i) < \varepsilon/2$. If $x \in X \subset
\widetilde{X}$, we have that $\widetilde{d}_n(x, \widetilde{x}_i)
< \varepsilon/2$, for some $\widetilde{x}_i \in \widetilde{G}$.
Hence it follows that
\[
d_n(x, x_i) \leq \widetilde{d}_n(x, \widetilde{x}_i) +
\widetilde{d}_n(x_i, \widetilde{x}_i) < \varepsilon/2 +
\varepsilon/2 = \varepsilon,
\]
showing that $G = \{x_1, \ldots, x_k\}$ is an $(n,
\varepsilon)$-generator set of $X$. If we choose $\widetilde{G}$
such that $\# \widetilde{G} = \widetilde{G}_n(\varepsilon/2)$,
then we get as claimed that
\[
G_n(\varepsilon) \leq \widetilde{G}_n(\varepsilon/2) < \infty.
\]

Now assume that $\alpha_{\delta} = \{B(x_1, \delta), \ldots,
B(x_k, \delta)\}$ is a covering of $X$, for every $\delta \in (a,
b)$, where $0 < a < b$. Since, for each fixed $\delta$, the number
of balls are finite, it follows that there exists
$\delta_{\varepsilon} \in (a, b)$ such that $\infty \notin
\widetilde{S}(x_i, \delta_{\varepsilon})$, for every $i \in \{1,
\ldots, k\}$, where $\widetilde{S}(x, r)$ is the sphere in
$\widetilde{X}$ of radius $r$ centered in $x$. Denoting by
$\widetilde{B}(x, r)$ the open ball in $\widetilde{X}$ of radius
$r$ centered in $x$, it remains two alternatives: 1) the point
$\infty$ is inside $\widetilde{B}(x_i, \delta_{\varepsilon})$ or
2) the point $\infty$ is not in the closure of $\widetilde{B}(x_i,
\delta_{\varepsilon})$. In the first case, the complement of
$B(x_i, \delta_{\varepsilon})$ in $X$ is equal to the complement
of $\widetilde{B}(x_i, \delta_{\varepsilon})$ in $X$, which is
compact. In the second alternative, there exist a open
neighborhood $U$ of $\infty$ which has empty intersection with
$\widetilde{B}(x_i, \delta_{\varepsilon})$. Thus $B(x_i,
\delta_{\varepsilon}) = \widetilde{B}(x_i, \delta_{\varepsilon})$
is in the complement of $U$ in $X$, which is compact. This shows
that the closure or the complement of $B(x_i,
\delta_{\varepsilon})$ in $X$ is compact, showing that
$\alpha_{\delta_{\varepsilon}}$ is already admissible.

We show now that every admissible covering of $X$ has a Lebesgue
number. If $\alpha$ is admissible covering of $X$, there exists an
open covering $\widetilde{\alpha}$ of $\widetilde{X}$ such that $A
\in \alpha$ if and only if there exists $\widetilde{A} \in
\widetilde{\alpha}$, with $A = \widetilde{A} \cap X$. In fact, if
$\alpha$ is an admissible covering of $X$, there exists at least
one $A \in \alpha$ with compact complement in $X$ and such that
its closure in $X$ is not compact. If we define $\widetilde{A} = A
\cup \{\infty\}$, we have that $\widetilde{A}$ is an open
neighborhood of $\{\infty\}$ in $\widetilde{X}$. Thus defining
\[
\widetilde{\alpha} = \left\{\widetilde{A}\right\} \cup \{B \in
\alpha : B \neq A\},
\]
it follows as claimed that $\widetilde{\alpha}$ is an open
covering of $\widetilde{X}$ such that $A \in \alpha$ if and only
if there exists $\widetilde{A} \in \widetilde{\alpha}$, with $A =
\widetilde{A} \cap X$. If $\varepsilon$ is a Lebesgue number for
$\widetilde{\alpha}$, then we claim that $\varepsilon$ is a also a
Lebesgue number for $\alpha$. To see this, if $B(x, \varepsilon)$
is a ball in $X$, there exists $\widetilde{A} \in
\widetilde{\alpha}$ such that $B(x, \varepsilon) \subset
\widetilde{A}$. Thus it follows that $B(x, \varepsilon) \subset
A$, where $A = \widetilde{A} \cap X \in \alpha$, completing the
proof that $d$ is an admissible metric.

Finally we show that $h^d(T) =
h^{\widetilde{d}}\left(\widetilde{T}\right)$. Let $G = \{x_1,
\ldots, x_k\}$ be  an $(n, \varepsilon/2)$-generator set of $X$.
By the density of $X$ in $\widetilde{X}$, if $\widetilde{x} \in
\widetilde{X}$, there exists $x \in X$ such that
$\widetilde{d}_n(\widetilde{x}, x) < \varepsilon/2$, since $d_n$
is topologically equivalent to $d$. Thus we have that $d_n(x, x_i)
< \varepsilon/2$, for some $x_i \in G$. Hence it follows that
\[
\widetilde{d}_n(\widetilde{x}, x_i) \leq
\widetilde{d}_n(\widetilde{x}, x) + d_n(x, x_i) < \varepsilon/2 +
\varepsilon/2 = \varepsilon,
\]
showing that $G = \{x_1, \ldots, x_k\}$ is an $(n,
\varepsilon)$-generator set of $\widetilde{X}$. Choosing $G$ such
that $\# G = G_n(\varepsilon/2)$, we get that
$\widetilde{G}_n(\varepsilon) \leq  G_n(\varepsilon/2)$. Since we
have also shown above that $G_n(\varepsilon) \leq
\widetilde{G}_n(\varepsilon/2)$, taking logarithms, dividing by
$n$ and taking limits, we get that
\[
g(\varepsilon) \leq \widetilde{g}(\varepsilon/2) \qquad \mbox{and}
\qquad \widetilde{g}(\varepsilon) \leq g(\varepsilon/2).
\]
Therefore it follows that
\[
h^d(T) = \lim_{\varepsilon \downarrow 0} g(4\varepsilon) \leq
\lim_{\varepsilon \downarrow 0} \widetilde{g}(2\varepsilon) =
h^{\widetilde{d}}\left(\widetilde{T}\right) \leq \lim_{\varepsilon
\downarrow 0} g(\varepsilon) = h^d(T),
\]
showing that $h^d(T) =
h^{\widetilde{d}}\left(\widetilde{T}\right)$. The last statement
now follows applying Proposition \ref{propequiv}.
\end{proof}

\section{The variational principle}

In this section, we present a full extension for non-compact sets
of the well known variational principle involving entropies. We
start with a proper map $T : X \to X$ and remember the concept of
entropy associated to some given $T$-invariant probability on $X$.
A \emph{$T$-invariant probability on $X$} is a Borel measure $\mu$
such that $\mu(X) = 1$ and $\mu(T^{-1}(A)) = \mu(A)$, for all
Borel subsets $A \subset X$. We denote by $\mathcal{P}_T(X)$ the
set of all $T$-invariant probability on $X$. A Borel partition
$\mathcal{A}$ of $X$ is a partition of $X$ where all of its
elements are Borel sets. Given a finite Borel partition
$\mathcal{A}$ of $X$, for every $n \in \mathbb{N}$, we have that
the set given by
\[
\mathcal{A}^n = \{A_0 \cap T^{-1}(A_1) \cap \ldots \cap
T^{-n}(A_n) : A_i \in \mathcal{A}\}
\]
is also a finite Borel partition of $X$, since $T$ is continuous.
For a given finite Borel partition $\mathcal{A}$ of $X$, we define
its associated $n$-entropy as
\[
H(\mathcal{A}^n) = \sum_{B \in \mathcal{A}^n} \phi(\mu(B)),
\]
where $\phi : [0, 1] \to \mathbb{R}$ is the continuous function
given by
\[
\phi(x) = \left\{
\begin{array}{cl}
-x\log(x), &  x \in (0, 1] \\
0, & x = 0
\end{array}
\right..
\]
It can be shown that the sequence $H(\mathcal{A}^n)$ is
subadditive, which implies the existence of the following limit
\[
h(T, \mathcal{A}) = \lim_{n \to \infty} \frac{1}{n}
H(\mathcal{A}^n).
\]
The \emph{$\mu$-entropy of the map $T$} is thus defined as
\[
h_{\mu}(T) = \sup_{\mathcal{A}} h(T, \mathcal{A}),
\]
where the supremum is taken over all finite Borel partition
$\mathcal{A}$ of $X$. The next result was first proved in the
Lemma 1.5 of \cite{handel}, where they used the fact that the
supremum of the measure entropies taken over all invariant
probabilities is equals to that one just taken over all ergodic
invariant probabilities. We present here an elementary proof of
the quoted result which dos not use this fact. In the following
$\mathcal{P}_T(X)$ denotes the set of all $T$-invariant
probabilities on $X$.

\begin{lema}\label{lema}
Let $T : X \to X$ be a proper map and $\widetilde{T} :
\widetilde{X} \to \widetilde{X}$ be its extension in the one point
compactification $\widetilde{X}$ of the locally compact separable
space $X$. Then it follows that
\[
\sup_{\mu} h_{\mu}(T) = \sup_{\widetilde{\mu}}
h_{\widetilde{\mu}}\left(\widetilde{T}\right),
\]
where the suprema are taken, respectively, over $\mathcal{P}_T(X)$
and over $\mathcal{P}_{\widetilde{T}}(\widetilde{X})$.
\end{lema}
\begin{proof}
If $\mu \in \mathcal{P}_T(X)$, definig
$\widetilde{\mu}\left(\widetilde{A}\right) =
\mu\left(\widetilde{A} \cap X\right)$, it is immediate that
$\widetilde{\mu} \in \mathcal{P}_{\widetilde{T}}(\widetilde{X})$,
since $X$ and $\{\infty\}$ are $\widetilde{T}$-invariant sets. It
is also immediate that $h_{\mu}(T) =
h_{\widetilde{\mu}}\left(\widetilde{T}\right)$, showing that
\[
\sup_{\mu} h_{\mu}(T) \leq \sup_{\widetilde{\mu}}
h_{\widetilde{\mu}}\left(\widetilde{T}\right).
\]

Now let $\widetilde{\mu} \in
\mathcal{P}_{\widetilde{T}}(\widetilde{X})$. If
$\widetilde{\mu}(\infty) = 1$, it is immediate that
$h_{\widetilde{\mu}}\left(\widetilde{T}\right) = 0$, since $X$ and
$\{\infty\}$ are $\widetilde{T}$-invariant sets. Thus we can
assume that $\widetilde{\mu}(\infty) = c < 1$. It is also
immediate that
\[
\mu = {\textstyle \left(\frac{1}{1 - c}\right)}
\widetilde{\mu}|_{\mathcal{B}(X)} \in \mathcal{P}_T(X).
\]
We claim that $h_{\widetilde{\mu}}\left(\widetilde{T}\right) \leq
h_{\mu}(T)$. If $\widetilde{\mathcal{A}} = \{\widetilde{A}_1,
\ldots, \widetilde{A}_l\}$ is a measurable partition of
$\widetilde{X}$, defining $A_i =  \widetilde{A}_i \cap X$, it
follows that $\mathcal{A} = \{A_1, \ldots, A_l\}$ is a measurable
partition of $X$. We have that, $B \in \mathcal{A}^n$ if and only
if there exists $\widetilde{B} \in \widetilde{\mathcal{A}}^n$,
with $B =  \widetilde{B} \cap X$. This follows, because
\[
\widetilde{T}^{-j}\left(\widetilde{A}_i\right) \cap X =
T^{-j}(A_i)
\]
for any $j \in \{0, \ldots, n\}$ and any $i \in \{1, \ldots, l\}$,
since $X$ and $\{\infty\}$ are $\widetilde{T}$-invariant sets. It
follows that
\[
H(\widetilde{\mathcal{A}}^n) = \sum_{\widetilde{B} \in
\widetilde{\mathcal{A}}^n}
\phi\left(\widetilde{\mu}\left(\widetilde{B}\right)\right) =
\phi\left(\widetilde{\mu}\left(\widetilde{B}_{\infty}\right)\right)
+ \sum_{\widetilde{B} \neq \widetilde{B}_{\infty}}
\phi\left(\widetilde{\mu}\left(\widetilde{B}\right)\right),
\]
where $\infty \in \widetilde{B}_{\infty}$. Since
$\widetilde{\mu}\left(\widetilde{B}\right) = (1 - c) \mu(B)$, for
each $\widetilde{B} \neq \widetilde{B}_{\infty}$, it follows that
\[
H(\widetilde{\mathcal{A}}^n) = b + \sum_{B \in \mathcal{A}^n}
\phi(a\mu(B)),
\]
where $a = 1 - c$ and $b =
\phi\left(\widetilde{\mu}\left(\widetilde{B}_{\infty}\right)\right)
- \phi(a\mu(B))$. Hence
\begin{eqnarray}
H(\widetilde{\mathcal{A}}^n) & = & b - \sum_{B \in \mathcal{A}^n} a\mu(B)\log(a\mu(B)) \nonumber \\
& = & b - a\left(\sum_{B \in \mathcal{A}^n} \mu(B)(\log(a) + \log(\mu(B))\right) \nonumber \\
& = & b - a\log(a)\left(\sum_{B \in \mathcal{A}^n} \mu(B)\right) + a\left(\sum_{B \in \mathcal{A}^n} \phi(\mu(B))\right) \nonumber \\
& = & b + \phi(a)+ a H(\mathcal{A}^n). \nonumber
\end{eqnarray}
It follows that $H(\widetilde{\mathcal{A}}^n) \leq d +
H(\mathcal{A}^n)$, since $a = 1 - c \leq 1$ and
\[
b + \phi(a) \leq d = 2\max\{\phi(x) : x \in [0, 1]\}.
\]
Dividing by $n$ and taking limits, we get that
\[
h_{\widetilde{\mu}}\left(\widetilde{T},
\widetilde{\mathcal{A}}\right) \leq h_{\mu}(T, \mathcal{A}) \leq
h_{\mu}(T).
\]
Since $\widetilde{\mathcal{A}}$ is arbitrary, we have that
\[
h_{\widetilde{\mu}}\left(\widetilde{T}\right) \leq h_{\mu}(T) \leq
\sup_{\mu} h_{\mu}(T).
\]
Since $\widetilde{\mu} \in
\mathcal{P}_{\widetilde{T}}(\widetilde{X})$ is also arbitrary, it
follows that
\[
\sup_{\widetilde{\mu}}
h_{\widetilde{\mu}}\left(\widetilde{T}\right) \leq \sup_{\mu}
h_{\mu}(T),
\]
completing the proof.
\end{proof}

Now we prove the variational principle for entropies in a locally
compact separable space. In Lemma 1.5 of \cite{handel}, it was
proved that
\[
\sup_{\mu} h_{\mu}(T) = \inf_{d} h_d(T).
\]
Here we improve this result showing that the above infimum is
always attained at $d$, whenever $d$ is an admissible metric.
Moreover the supremum and the minimum coincide with the
topological entropy introduced in the Section \ref{sectopent}.

\begin{teorema}
Let $T : X \to X$ be a continuous map, where $X$ is a locally
compact separable space. Then it follows that
\[
\sup_{\mu} h_{\mu}(T) = h(T) = \min_{d} h_d(T),
\]
where the minimum is attained whenever $d$ is an admissible
metric.
\end{teorema}
\begin{proof}
The the variational principle for compact spaces states that
\[
h\left(\widetilde{T}\right) = \sup_{\widetilde{\mu}}
h_{\widetilde{\mu}}\left(\widetilde{T}\right).
\]
By Proposition 1.4 of \cite{handel}, we have that
\[
\sup_{\mu} h_{\mu}(T) \leq \inf_{d} h_d(T).
\]
Applying Lemma \ref{lema}, it follows that
\[
h\left(\widetilde{T}\right) = \sup_{\mu} h_{\mu}(T) \leq \inf_{d}
h_d(T).
\]
Applying Propositions \ref{propequiv} and \ref{propadmissivel}, we
get that
\[
h_d(T) \leq h^d(T) = h(T) = \sup_{\mu} h_{\mu}(T) \leq \inf_{d}
h_d(T),
\]
where, in the first two terms, $d$ is any admissible metric.
\end{proof}

\section{Topological entropy of automorphisms}

In this section, we compute the topological entropy for
automorphisms of simply connected nilpotent Lie groups. We start
with linear isomorphisms of a finite dimensional vector space. For
this, we need to determine the recurrent set of a linear
isomorphism in terms of its multiplicative Jordan decomposition
(see Lema 7.1, page 430, of \cite{helgason}). If $T : V \to V$ is
a linear isomorphism, where $V$ is a finite dimensional vector
space, then we can write $T = T_HT_ET_U$, where $T_H : V \to V$ is
diagonalizable in $V$ with positive eigenvalues, $T_E : V \to V$
is an isometry relative to some appropriate norm and $T_U : V \to
V$ is a linear isomorphism which can be decomposed into the sum of
the identity map plus some nilpotent linear map. The linear maps
$T_H$, $T_E$ and $T_U$ commute, are unique and called,
respectively, the \emph{hyperbolic}, the \emph{elliptic} and the
\emph{unipotent} components of the multiplicative Jordan
decomposition of $T$. In the next result, we prove that the
recurrent set $R(T)$ of $T$ is given as the intersection of the
fixed points of the hyperbolic and unipotent components. Using
this characterization, we also show that the topological entropy
of $T$ always vanishes.

\begin{proposicao}\label{proplinear}
Let $T : V \to V$ be a linear isomorphism, where $V$ is a finite
dimensional vector space. Then the recurrent set of $T$ is given
by
\[
R(T) = \emph{fix}\left(T_H\right) \cap \emph{fix}\left(T_U\right).
\]
Furthermore, it follows that $h(T) = 0$.
\end{proposicao}
\begin{proof}
Let $\mathbb{P}T : \mathbb{P}V \to \mathbb{P}V$ be the map induced
by $T$ in the projective space of $V$ and, for a subspace $W
\subset V$, denote by $\mathbb{P}W$ its projection in
$\mathbb{P}V$. By Proposition 2 of \cite{fpss}, we have that
\[
R(\mathbb{P}T) = \mbox{fix}\left(\mathbb{P}T_H\right) \cap
\mbox{fix}\left(\mathbb{P}T_U\right).
\]
By linearity, it is immediate that $\mathbb{P} \langle R(T)\rangle
\subset R(\mathbb{P}T)$, where $\langle A \rangle$ denotes the
linear subspace generated by $A \subset V$. Thus it follows that
\[
R(T) \subset \mbox{eig}(T_H) \cap \mbox{eig}(T_U),
\]
where $\mbox{eig}(S)$ denotes the union of the eigenspaces of a
linear map $S : V \to V$. Since $T_U$ is unipotent, all of its
eigenvalues are equals to one, implying that $\mbox{eig}(T_U) =
\mbox{fix}(T_U)$. Hence $R(T) \subset \mbox{eig}(T_H) \cap
\mbox{fix}(T_U)$. Now let $v \in R(T)$ be such that $T_H v =
\lambda v$, for some $\lambda > 0$. Since the multiplicative
Jordan decomposition is commutative, it follows that
\[
|T^n v| = |T_E^n T_H^n T_U^n v| = |T_H^n v| = \lambda^n |v|,
\]
where we used the fact that $T_E$ is an isometry relative to some
appropriate norm $|\cdot|$. This shows that $\lambda$ is equals to
one and thus that $R(T)$ is contained in $\mbox{fix}(T_H) \cap
\mbox{fix}(T_U)$. Since $\mbox{fix}(T_H) \cap \mbox{fix}(T_U)$ is
invariant by $T_E$, we can consider the restriction of $T$ to this
set, which is just the restriction of $T_E$. Thus we get that
$R(T) = \mbox{fix}(T_H) \cap \mbox{fix}(T_U)$, since the
restriction of $T_E$ to $\mbox{fix}(T_H) \cap \mbox{fix}(T_U)$ is
an isometry whose orbits have compact closure.

Now to show that $h(T) = 0$, we first note that, by the Poincar{\'e}
recurrence theorem, for every $\mu \in \mathcal{P}_T(X)$, we have
that $h_{\mu}(T) = h_{\mu}\left(T|_{R(T)}\right)$, since $R(T)$ is
a closed set. By the variational principle, it follows that
\[
h(T) = h\left(T|_{R(T)}\right) \leq h_d\left(T|_{R(T)}\right),
\]
for every metric $d$. Since $T|_{R(T)} = T_E|_{R(T)}$ is an
isometry relative to some appropriate metric $d$, we get that
$h_d\left(T|_{R(T)}\right) = 0$, completing the proof.
\end{proof}

We note that the classical formula for the $d$-entropy of a linear
isomorphism, where $d$ is the euclidian metric, is just an upper
bound for its null topological entropy. Thus it might be an
interesting problem to calculate the topological entropy of a
given automorphism $\phi$ of a noncompact Lie group $G$, since the
classical formula for its $d$-entropy, where $d$ is some invariant
metric, is as well just an upper bound for its topological
entropy. Remember that the classical formula is given by
\[
h_d(\phi) =
\sum_{|\lambda| > 1} \log |\lambda|,
\]
where $\lambda$ runs through the eigenvalues of $\mbox{d}_1\phi :
\mathfrak{g} \to \mathfrak{g}$, the differential of the
automorphism $\phi$ at the identity element of $G$.

The following theorem gives an answer for the above problem when
$G$ is a simply connected nilpotent Lie group.

\begin{teorema}
Let $\phi : G \to G$ be an automorphism, where $G$ is a simply
connected nilpotent Lie group. Then it follows that $h(\phi) = 0$.
\end{teorema}
\begin{proof}
If $G$ is a connected and  simply connected nilpotent Lie group,
we have that the exponential map is a diffeomorphism between
$\mathfrak{g}$ and $G$ (see Theorem 1.127, page 107, of
\cite{knapp}). Since $\phi(\exp(X)) = \exp(\mbox{d}_1\phi X)$, we
get that $\phi$ and $\mbox{d}_1\phi$ are conjugated maps. By
Proposition \ref{propsemiconjugacao}, it follows that $h(\phi) =
h(\mbox{d}_1\phi) = 0$.
\end{proof}

We note that, if the fundamental group of $G$ is not trivial, it
is possible the existence of an automorphism with positive
topological entropy, even when $G$ is abelian. In fact, the map
$\phi(z) = z^2$ is an automorphism of the abelian Lie group $S^1$
and it is well known that it has topological entropy $h(\phi) =
\log(2) > 0$. We have that the canonical homomorphism $\pi :
\mathbb{R} \to S^1$, given by $\pi(x) = e^{i x}$, is a
semi-conjugation between $\phi$ and the automorphism
$\widetilde{\phi}$ of the universal covering $\mathbb{R}$ of
$S^1$, given by $\widetilde{\phi}(x) = 2x$. Although
$h\left(\widetilde{\phi}\right) = 0$, we can not apply Proposition
\ref{propsemiconjugacao} to conclude that $h(\phi) = 0$, since the
canonical homomorphism $\pi$ is a proper map if and only if the
fundamental group of $G$ is finite.

\end{document}